\documentclass[11pt]{article}

\usepackage{amsbsy,amsfonts,amsmath,amssymb,mathrsfs}
\usepackage[T1]{fontenc}
\usepackage{tikz}
\usetikzlibrary{cd,arrows,matrix,positioning}
\usepackage{hyperref}
\usepackage{cleveref}
\usepackage{comment}
\usepackage{baskervald}
\usepackage{color}

\makeatletter
\newcommand*{\defeq}{\mathrel{\rlap{%
                     \raisebox{0.3ex}{$\m@th\cdot$}}%
                     \raisebox{-0.3ex}{$\m@th\cdot$}}%
                     =}
\makeatother

\newtheorem{counter}[subsubsection]{$\!\!$}
\newtheorem{subcounter}[subsection]{$\!\!$}

\newenvironment{defi*}{\begin{subcounter} \rm {\bf Definition.}}{\end{subcounter}}

\newenvironment{prop}{\begin{counter} {\bf Proposition.}}{\end{counter}}
\newenvironment{prop*}{\begin{subcounter} {\bf Proposition.}}{\end{subcounter}}
\newenvironment{lemma}{\begin{counter} {\bf Lemma.}}{\end{counter}}
\newenvironment{lemma*}{\begin{subcounter} {\bf Lemma.}}{\end{subcounter}}

\newenvironment{coro*}{\begin{subcounter} {\bf Corollary.}}{\end{subcounter}}

\newenvironment{theo*}{\begin{subcounter} {\bf Theorem.}}{\end{subcounter}}

\newenvironment{rema}{\begin{counter} \rm {\bf Remark.}}{\end{counter}}
\newenvironment{rema*}{\begin{subcounter} \rm {\bf Remark.}}{\end{subcounter}}

\newenvironment{exam*}{\begin{subcounter} \rm {\bf Example.}}{\end{subcounter}}

\newenvironment{noth*}[1]{\begin{subcounter} {\bf #1.}\rm}{\end{subcounter}}
\newenvironment{proof}{{\flushleft \bf Proof~:}}{\hfill $\square$ \vspace{5mm}}

\DeclareMathOperator{\id}{id}

\def\sX{\mathscr{X}}  

\begin{document}

\begin{center}
{\bf \Large Maximal models of torsors over a local field}

\bigskip
\bigskip

Yuliang Huang

\bigskip
\today
\end{center}

\bigskip

\begin{center}
\begin{minipage}{17.4cm}
{\small {\bf Abstract.} Let $R$ be a discrete valuation ring, and $K$ its fraction field. In 1967, Raynaud initiated the notion of maximal $R$-model for torsors over $K$, and it was further developed by Lewin-M\'en\'egaux. In this paper, motivated by a conjectural ramification theory for infinitesimal torsors, we investigate this notion of maximal model in greater detail. We prove the maximality, the compatibility along inductions, and an existence result for group schemes of semi-direct products.}
\end{minipage}
\end{center}
2010 Mathematics Subject Classification:
11S15,
14G20,
14L30,

\noindent Keywords: maximal model, torsor, abelian scheme, finite flat group scheme, different.

\tableofcontents

\newpage

\section{Introduction}

In the book of Serre \cite{Se79}, the classical ramification theory of local fields is well-documented. The assumption of residue field being perfect is crucial in this classical theory. In the need of pursuing ramification theory for higher dimensional schemes, Abbes and Saito (cf. \cite{AS02}, \cite{AS03}, \cite{Sa12}) have developed the ramification theory for local fields with possibly imperfect residue field, using techniques of rigid analytic geometry. Later on, Saito \cite{Sa19} also gave a schematic approach. This more general theory is adapted in the later works (cf. \cite{KS08}, \cite{KS13}) of Kato and Saito on ramification theory of higher dimensional varieties over perfect fields and local fields. \\

Ramification of finite \textit{separable} extensions of local fields with possibly imperfect residue field is well-understood by now, thanks to the theory of Abbes--Saito. The motivation of this paper is, what about \textit{arbitrary} finite extensions of local fields? For example, if the local field has positive characteristic, there are many inseparable extensions which cannot be recognized in the absolute Galois group. Do they have a reasonable theory of ramification? In order to talk about this question, we have to make certain things clear. First, what does it mean by ``Galois'' for inseparable extensions? In the classical setting, the Galois group $G_{L/K}$ of a Galois extension $L/K$ of local fields acts on $L$, and it makes the scheme ${\rm Spec}(L)$ a $G_{L/K}$-torsor over ${\rm Spec}(K)$. For a purely inseparable extension $K'/K$, there is no nontrivial abstract group acting on $K'$ fixing $K$. Instead, there could be some infinitesimal (or called local) group scheme acting on $K'$, and making ${\rm Spec}(K')$ a torsor over ${\rm Spec}(K)$. A significant difference is that, there might be more than one group scheme acting on $K'$, and in particular, it is reasonable to expect that they could have different behaviors of ramification. \\

Many attempts have been studied, especially, the notion of \textit{tame action} by group schemes has been greatly developed by various people, cf. \cite{CEPT96}. In \cite{Za16}, Zalamansky proved an analogous Riemann--Hurwitz formula for inseparable covers under infinitesimal diagonalizable group schemes. For more general case, there still needs new ideas. \\

In the attempt of finding such a more general ramification including infinitesimal torsors, our basic setting is as follows. Let $K$ be a local field, $\mathcal{O}_K$ the ring of integers, and $k$ the residue field. Let $G$ be a group scheme over $\mathcal{O}_K$, and then we start with a $G_K$-torsor $P_K$ over ${\rm Spec}(K)$. The first problem we are facing is that we need a generalization of the notion of integer rings in the classical theory, in other words, we need to find an integral model $P$ of $P_K$ with extended $G$-action. In the classical case, this is standard by taking the integral closure of $\mathcal{O}_K$ in the extension field, with the natural action by Galois group. In our general setting, a good candidate is Lewin-M\'en\'egaux's notion of ``minimal models'' of torsors over $K$. The study of Lewin-M\'en\'egaux's ``minimal model'' is our main interest of this paper. \\

The construction of ``minimal model'' originated in Raynaud's paper \cite{Ra67}, page 82--83, where he did the construction for abelian schemes, and intended to claim that it was {\it minimal} in the usual sense:
\begin{quote}
    {\it ``c) Soient $A$ un anneau de valuation discr\`ete, $K$ son corps de fractions, $G$ une $K$-vari\'et\'e ab\'elienne qui poss\`ede une bonne r\'eduction sur $A$, de sorte que $G$ se prolonge en un sch\'ema ab\'elien $\mathscr{G}$ sur $S={\rm Spec}~A$ et soit $X$ un $K$-espace principal homog\`ene sous $G$. Montrons que $X$ se prolonge en un $S$-sch\'ema {\rm projectif} et {\rm r\'egulier} $\sX$ (qui sera d'ailleurs un mod\`ele minimal de $X$ dans la terminologie de [2]). ...''}
\end{quote}
where [2] is \cite{Ne67} in the same volume. Later on, the idea was further developed by Lewin-M\'en\'egaux in her paper \cite{LM83}. But unfortunately, her paper is very sketchy and left ``minimality'' unexplained. It turns out surprisingly, a ``minimal model'' is not minimal but on contrary {\it maximal}, among all the integral models (see Definition \ref{def:integral_model}, Definition \ref{def:max_model}).

\begin{theo*}{\rm [Theorem \ref{thm:max}]}
Let $G$ be a flat $S$-group scheme of finite type, and $X_K\to {\rm Spec}(K)$ a $G_K$-torsor. We assume that the maximal model $X\to S$ exists. Then $X$ is maximal among all the integral models of $X_K\to{\rm Spec}(K)$, namely, if $\mathcal{X}$ is another integral model, then there is a unique model morphism $X\to \mathcal{X}$.
\end{theo*}

\noindent As our result suggested, we reasonably change the terminology to \textit{maximal model}. Maximal model has some pleasant features. For example, it is unique up to isomorphisms if exists,\footnote{The uniqueness is not merely a consequence of maximality, it is true for maximal models over arbitrary flat $S$-schemes, provided existence. See Proposition \ref{prop:max_unique}. On the other hand, we only have maximality over $S$.} and it recovers the integer ring in the classical situation, i.e., if $L/K$ is a Galois extension with the Galois group $G_{L/K}$, then the maximal model of the $G_{L/K}$-torsor ${\rm Spec}(L)$ is exactly ${\rm Spec}(\mathcal{O}_L)$, where $\mathcal{O}_L$ is the ring of integers of $L$. Under a condition of extension property, maximal models are functorial with respect to group schemes.

\begin{theo*}{\rm [Theorem \ref{thm:functoriality}]}
Let $\varphi:G\to H$ be a homomorphism of flat $S$-group schemes of finite type, and $f_K:X_K\to Y_K$ a $G_K$-equivariant morphism from a $G_K$-torsor $X_K$ to a $H_K$-torsor $Y_K$. Suppose that the maximal models $X,Y$ of $X_K,Y_K$ exist. Moreover, suppose that the following extension property holds:
\begin{enumerate}
    \item[(E)] for some finite flat base change $S'/S$ coming from a finite extension $K'/K$ of fields, which verifies the definition of maximal model for $X,Y$, there exists a section $a\in X(S')$ such that $(f_K\otimes K')(a_{K'})$ extends to a section $b\in Y(S')$.
\end{enumerate}
Then $f_K$ extends to a $G$-equivariant morphism $f$ which fits into the commutative diagram
\[
\begin{tikzcd}
G_{S'} \arrow[r, "\varphi_{S'}"] \arrow[d, "\pi_a"'] & H_{S'} \arrow[d, "\pi_b"] \\
X \arrow[r, "f"] & Y
\end{tikzcd}
\]
where $\pi_a,\pi_b$ are the finite flat morphisms as in the definition of maximal model, constructed from the sections $a$ and $b$. Moreover, if $\varphi$ is faithfully flat, then so is $f$.
\end{theo*}

More importantly, maximal models are compatible with inductions, and in particular they are compatible with quotients.

\begin{theo*}{\rm [Theorem \ref{thm:induction}]}
Let $\varphi:G\to H$ be a homomorphism of flat $S$-group schemes, $X_K$ is a $G_K$-torsor. Assume that the maximal model $X$ of $X_K$ exists. If the induced $H_K$-torsor $Y_K$ has a maximal model $Y$, then ${\rm Ind}_G^H X = X\times^G H$ is representable by $Y$.
\end{theo*}
This property is very necessary and important for building general ramification theory using maximal models. \\

As for the existence of maximal model, we have the following result concerning maximal models under semi-direct products:

\begin{theo*}{\rm [Theorem \ref{theorem:semi-direct_prod}]}
Let $X_K$ be an $G_K$-torsor over $K$, and $Y_K={\rm Ind}_{G_K}^{H_K}X_K$ the induced $H_K$-torsor. If the maximal model $Y$ of $Y_K$ exists, then the following are equivalent
\begin{enumerate}
    \item[(1)] The maximal model $X$ of $X_K$ exists;
    \item[(2)] The maximal model $X\to Y$ of the $N_{Y_K}$-torsor $X_K\to Y_K$ exists.
\end{enumerate}
Moreover, $X$ is the common maximal model in (1) and (2), if it exists.
\end{theo*}

\noindent {\bf Overview}. In Lewin-M\'en\'egaux's paper \cite{LM83}, she has proved uniqueness, (only claimed) functoriality of maximal model, the existence for the case of abelian schemes and finite flat commutative group schemes, and property of the different. However, her paper is highly sketchy, we give full details and represent her results in Section \ref{existence}, \ref{different}, and we prove functoriality in Section \ref{max}. Besides, we prove the maximality, and the compatibility with inductions in Section \ref{max}. In Section \ref{semi-direct_prod}, we prove a result on the existence of maximal model of torsors under a semi-direct product of group schemes. In Section \ref{sec:max_example}, we study and classify maximal models of torsors under some finite flat group schemes of order $p$. \\

\noindent {\bf Acknowledgement}. The author wishes to thank his PhD supervisor Matthieu Romagny, for introducing the topic and his constant support during the work. He also would like to thank Jo\~{a}o Pedro dos Santos, Laurent Moret-Bailly, and Dajano Tossici, for their interest on this work, and very valuable remarks and suggestions. \\

\section{Basic definitions (general setting)}

Throughout the paper, we work with the category of \textit{separated} schemes. When we work over a DVR base $S=\text{Spec}(R)$, we always assume that $R$ is \textit{henselian} and \textit{Japanese}. Recall that, being Japanese means that for any finite extension $K'$ of the fraction field $K$ of $R$, the normalization $R'$ of $R$ is a finite $R$-module. Torsors under group schemes are required to be schemes by definition, unless mention of the contrary.  \\

Let $R$ be a discrete valuation ring, with residue field $k$ and fraction field $K$. Let $S={\rm Spec}(R)$. If $Y$ is a $S$-scheme, then we always let $Y_K$ and $Y_k$ be its generic fiber and special fiber respectively. \\

Let $T$ be a $S$-scheme. Let $G$ be a flat $S$-group scheme of finite type, and $X_K\to T_K$ a $G_K$-torsor. First let us define the notion of \textit{integral model} of such a torsor.

\begin{defi*}\label{def:integral_model}
Let $X_K\to T_K$ be a $G_K$-torsor. An {\it integral $S$-model} (or {\it integral $R$-model}) of $X_K\to T_K$ is a faithfully flat separated morphism $\mathcal{X}\to T$ of finite type whose generic fiber is isomorphic to $X_K\to T_K$, together with an extended $G$-action on $\mathcal{X}$. A {\it model morphism} between integral models of $X_K\to T_K$ is a $G$-equivariant morphism which induces identity on generic fibers. \\
\end{defi*}

Among all the integral models, we are interested in finding a convenient one. In this paper, we study the following notion of maximal model, which serves as a candidate.

\begin{defi*}\label{def:max_model}
A {\it maximal $S$-model} (or {\it maximal $R$-model}) of the $G_K$-torsor $X_K$ is an integral model $X\to T$ of $X_K\to T_K$ satisfying the following condition: there exists a finite extension $K'/K$ and a finite flat $G$-equivariant morphism $f:G\times_S T_{S'}\to X$, where $S'$ denotes the normalization of $S$ in $K'$.
\end{defi*}

\begin{rema*} \label{rmk:max_model}
\begin{enumerate}
  \item[(1)] Notice that the original notion in \cite{LM83} is called ``minimal model'', where the idea goes back to Raynaud \cite{Ra67}. We have found it awkward because later we will see that such a ``minimal'' model is actually maximal. To avoid further ambiguity, we decided to correct the terminology.
  \item[(2)] From the last condition, we have the diagram
      \[
      \begin{tikzcd}
      G_{T'} \arrow[rdd, bend right] \arrow[rd, "g" description] \arrow[rrd, bend left, "f"] & & \\
      & X_{T'} \arrow[r] \arrow[d] & X \arrow[d]\\
      & T' \arrow[r] & T
      \end{tikzcd}
      \]
      where $T':=T\times_S S'$, and $f$ factors through the $T'$-morphism $g$. In particular, the morphism $g_{K'}$ on the generic fiber is an isomorphism of trivial $T'$-torsors, since $K'$ trivializes $X_K\to T_K$.
  \item[(3)] Note that it is meaningless to speak of maximal model of a $G_K$-torsor $X_K\to T_K$ which cannot be trivialized by any finite extension $K'/K$. If the torsor can be trivialized by a finite extension $K'/K$, and it naturally extends to a $G$-torsor $X\to T$ which is trivialized by the normalization $S'/S$, then this $G$-torsor is automatically a maximal model of $X_K\to T_K$. Indeed, according to the previous diagram
      \[
      \begin{tikzcd}
      G_{T_{S'}} \arrow[rrd, bend left, "f"] \arrow[rd, "\simeq" description] \arrow[rdd, bend right] & & \\
      & X_{T_{S'}} \arrow[r] \arrow[d] & X \arrow[d]\\
      & T_{S'} \arrow[r] & T
      \end{tikzcd}
      \]
      the morphism $f$ is isomorphic to the projection $X_{T_{S'}}\to X$, which is finite flat since $S'/S$ is finite flat. Note that here if the ring $R$ is not Japanese, it may very well happen that such extension $S'/S$ is not finite, hence it does not verify the definition of maximal model. \\
\end{enumerate}
\end{rema*}

\begin{prop*} \label{prop:max_unique}
A maximal model $X$ of $X_K$, if it exists, is unique up to unique $T$-isomorphism.
\end{prop*}

\begin{proof}
Let $X_1$ and $X_2$ be maximal $S$-models of $X_K$. Indeed, we can choose the same finite extension $K'_1=K'_2=K'/K$ which trivializes the torsor $X_K$. Let $T'=T\times_S S'$ where $S'$ is the normalization of $S$ in $K'$, and denote the morphisms of two models from $G_{T'}$ by
\[
f_i:\ G_{T'} := G\times_S T'\ \longrightarrow\ X_i,\ \ \ \ \ \ i=1,2.
\]
The finite flat $T$-morphism $f_1$ (resp. $f_2$) realizes $X_1$ (resp. $X_2$) as the cokernel of the finite flat groupoids
\[
R_1:=G_{T'}\times_{X_1}G_{T'}\ \rightrightarrows\ G_{T'},
\]
(resp. $R_2$). It is clear that $R_1|_{T'_K}\cong R_2|_{T'_K}$. The morphism $f_2:G_{T'}\to X_2$ is both $R_1$- and $R_2$-invariant, since it is invariant on the generic fiber, which implies the invariance by flatness of the groupoid. Therefore it induces a $T$-morphism $\alpha:X_1\to X_2$ such that $f_2\circ\alpha = f_1$. Similarly, there is a $T$-morphism $\beta:X_2\to X_1$ with $f_1\circ\beta = f_2$, and we have
\[
\left\{\begin{array}{l}
f_1 = f_2\circ\alpha = f_1\circ(\beta\alpha) \\
f_2 = f_1\circ\beta = f_2\circ(\alpha\beta)
\end{array}\right.
\]
which indicates that $\beta\alpha = {\rm id}_{X_1}$ and $\alpha\beta = {\rm id}_{X_2}$, because $f_1,f_2$ are epimorphisms.
\end{proof}

\begin{exam*}
Let $G$ be a constant finite group, and $K'/K$ a finite Galois extension with Galois group $G$. Then ${\rm Spec}(K')\to{\rm Spec}(K)$ is a $G_K$-torsor, and it is trivialized by the field extension $K'/K$. Let $R'\subset K'$ be the normalization of $R$ in $K'$, then ${\rm Spec}(R')\to{\rm Spec}(R)$ is a maximal $R$-model. Indeed, the finite flat $G$-equivariant morphism is given by projection. In particular, the maximal model ${\rm Spec}(R')$ remains to be a $G$-torsor if and only if the extension $K'/K$ of local fields is unramified. \\
\end{exam*}

\begin{exam*}
Let $K=k(\!(t)\!)$, where $k$ is a field of characteristic $p>0$. Consider the purely inseparable extension $L=k(\!(t^{1/p})\!)$ as an $\alpha_{p,K}$-torsor over $K$
\[
{\rm Spec}(L):= {\rm Spec}(k(\!(t^{1/p})\!)) \longrightarrow {\rm Spec}(K),
\]
where the $\alpha_{p,K}$-action is given by \footnote{It is the Frobenius of $\mathbb{P}^1_k$ at $\infty$.}
\[
t^{1/p} \longmapsto \frac{t^{1/p}}{1+at^{1/p}},
\]
here $a$ is the coordinate of $\alpha_{p,K}$. The action naturally extends to the integer rings $\mathcal{O}_L$ of $L$, and the $\mathcal{O}_K$-scheme ${\rm Spec}(\mathcal{O}_L)$ is the maximal model of ${\rm Spec}(L)\to{\rm Spec}(K)$. Indeed, the $\alpha_{p,K}$-torsor is tautologically trivialized by $L/K$, and we have the diagram
\[
\begin{tikzcd}
\alpha_{p,\mathcal{O}_L} \arrow[rrd, bend left=15, "f"] \arrow[rd, dashed, "g" description] \arrow[rdd, bend right=15] & & \\
& {\rm Spec}(\mathcal{O}_L\otimes_{\mathcal{O}_K}\mathcal{O}_L) \arrow[r] \arrow[d] & {\rm Spec}(\mathcal{O}_L) \arrow[d]\\
& {\rm Spec}(\mathcal{O}_L) \arrow[r] & {\rm Spec}(\mathcal{O}_K)
\end{tikzcd}
\]
where $f$ is clearly finite flat. Here the maximal model is not a torsor, hence it is ``ramified''. \\
\end{exam*}

\section{Maximality and functoriality of maximal models}\label{max}

\noindent From now on, we restrict ourselves to maximal models of torsors over a discrete valuation ring $S={\rm Spec}(R)$. The maximality in the name of maximal model agrees with the usual sense that for any integral $S$-model, there is a dominant morphism from the maximal $S$-model to such model. \\

\begin{exam*}
Let $K=k(\!(t)\!)$ be a local field of equicharacteristic where $k$ has characteristic $p>0$, its ring of integers is $R=k[\![t]\!]$. Consider the constant group $G=\mathbb{Z}/p\mathbb{Z}$ which we view as a constant group scheme over $R$, and $P_K\to{\rm Spec}(K)$ the trivial $G_K$-torsor. Obviously its maximal model is the trivial $G$-torsor $P=\mathbb{Z}/p\mathbb{Z}\to{\rm Spec}(R)$. Yet we have another integral model $\mathcal{P}={\rm Spec}(R[a]/(a^p-t^{p-1}a))$, where the action of $G={\rm Spec}(R[e]/(e^p-e))$ is given by
\[
a \longmapsto a+et.
\]
It is straightforward that this is an integral model of the original $G_K$-torsor, but it is not a maximal $S$-model. There is a natural morphism from the maximal model $\mathbb{Z}/p\mathbb{Z}$ to $\mathcal{P}$
\begin{eqnarray*}
R[a]/(a^p-t^{p-1}a) & \longrightarrow & R[e]/(e^p-e) \\
a & \longmapsto & et
\end{eqnarray*}
which is a {\it full set of sections} in the sense of Katz--Mazur \cite{KM85}. \\
\end{exam*}

\begin{theo*}\label{thm:max}
Let $X_K\to {\rm Spec}(K)$ be a $G_K$-torsor, we assume that the maximal model $X\to S$ exists. Then $X$ is maximal among all the integral models of $X_K\to{\rm Spec}(K)$, namely, if $\mathcal{X}$ is another integral model, then there is a unique model morphism $X\to \mathcal{X}$.
\end{theo*}

\begin{proof}
First, we claim that there exists a section $y\in\mathcal{X}(S')$ for some finite flat extension $S'/S$, where $S'$ is the normalization of some finite extension of fraction field $K$. Indeed, since $\mathcal{X}/S$ is surjective, we choose a closed point $x_0$ of the special fiber $\mathcal{X}_k$. Because its local ring $\mathscr{O}_{\mathcal{X},x_0}$ is flat over $S$, there is a generization $x_1$ of $x_0$. The schematic closure $\overline{\{x_1\}}$ is irreducible and faithfully flat over $S$. By Proposition 10.1.36 in \cite{Liu02}, there exists a closed point $x_1'$ of $\overline{\{x_1\}}\otimes K$, such that $x_0$ is a specialization of $x_1'$. The residue field $K'$ of $x_1'$ is a finite extension of $K$, the normalization $S'$ of $S$ in $K'$ is a DVR, since $S$ is henselian. Let $D$ denote the schematic closure of $x_1'$, which is naturally an $S'$-scheme. Note that $D$ is a subscheme of the separated scheme $\mathcal{X}_{S'}$, hence $D$ is also separated. In other words, $x_0$ is the only specialization of $x_1'$, i.e., $D$ only has two closed points $x_0$ and $x_1'$. In particular, the structral morphism $D\to S'$ is a homeomorphism. By \cite{SP19} Tag~\href{http://stacks.math.columbia.edu/tag/04DE}{04DE}, $D$ is affine, and we let $A$ be its algebra of global functions. From the structral morphism $D\to S'$
\[
\begin{tikzcd}
\mathscr{O}_{S'} \arrow[r] \arrow[rd, hook] & A \arrow[d, hook] \\
& K
\end{tikzcd}
\]
we know that $\mathscr{O}_{S'}\subset A$ in $K$, which forces $\mathscr{O}_{S'}\simeq A$ since $A$ is a local ring. Thus $D$ provides a section in $\mathcal{X}(S')$. \\

Let $S'\to S$ be a finite flat base change via some extension of fraction field $K'/K$, which trivializes the $G_K$-torsor $X_K$ and gives rise to a finite flat morphism $G_{S'}\to X$. This is guaranteed by the definition. Moreover by our previous claim, we may assume that there exists a section $y\in \mathcal{X}(S')$. Consider the finite flat groupoid
\[
\Gamma_X := G_{S'}\times_X G_{S'} \rightrightarrows G_{S'} \rightarrow X,
\]
and note that the morphism $G_{S'}\to X$ is effectively epimorphic. The integral point $y$ also gives rise to a groupoid
\[
\Gamma_{\mathcal{X}} := G_{S'}\times_{\mathcal{X}} G_{S'} \rightrightarrows G_{S'} \rightarrow \mathcal{X}
\]
where the map $G_{S'}\to\mathcal{X}$ is contructed via
\[
\begin{tikzcd}
G_{S'}\simeq G_{S'}\times_{S'}S' \arrow[r, "\id \times y"] & G_{S'}\times_{S'}\mathcal{X}_{S'} \arrow[r] & \mathcal{X}_{S'} \arrow[r] & \mathcal{X}
\end{tikzcd}
\]
The generic fibers of above two groupoids are isomorphic. Notice that $\Gamma_X$ and $\Gamma_{\mathcal{X}}$ are contained in the same background scheme $G_{S'}\times_S G_{S'}$. Since $\Gamma_X$ is flat over $S$, it is the flat closure of the generic fiber, hence we have a closed immersion
\[
\Gamma(X)\simeq \text{schematic closure of $\Gamma(\mathcal{X})\otimes K$}\hookrightarrow \Gamma(\mathcal{X}).
\]
Now the two compositions
\[
\Gamma_X \rightrightarrows G_{S'} \rightarrow \mathcal{X}
\]
coincide with the restriction of the compositions $\Gamma_{\mathcal{X}}\rightrightarrows G_{S'}\to \mathcal{X}$ to $\Gamma_X$, hence they agree on $\Gamma_X$. By the fact that $G_{S'}\to X$ is an effective epimorphism, it induces a unique morphism from $X$ to $\mathcal{X}$
\[
\begin{tikzcd}
\Gamma_X \arrow[r, shift left] \arrow[r, shift right] \arrow[d, hook] & G_{S'} \arrow[r] \arrow[d, equal] & X \arrow[d, dashed] & \\
\Gamma_{\mathcal{X}} \arrow[r, shift left] \arrow[r, shift right] & G_{S'} \arrow[r] & \mathcal{X}
\end{tikzcd}
\]
which is an isomorphism on the generic fibers.
\end{proof}

Now let us study the functorial behavior of maximal models. We do not have functoriality for maximal models in full generality, but only under certain condition.

\begin{theo*}\label{thm:functoriality}
Let $\varphi:G\to H$ be a homomorphism of flat $S$-group schemes of finite type, and $f_K:X_K\to Y_K$ a $G_K$-equivariant morphism from a $G_K$-torsor $X_K$ to a $H_K$-torsor $Y_K$. Suppose that the maximal models $X,Y$ of $X_K,Y_K$ exist. Moreover, suppose that the following extension property holds:
\begin{enumerate}
    \item[(E)] for some finite flat base change $S'/S$ coming from a finite extension $K'/K$ of fields, which verifies the definition of maximal model for $X,Y$, there exists a section $a\in X(S')$ such that $(f_K\otimes K')(a_{K'})$ extends to a section $b\in Y(S')$.
\end{enumerate}
Then $f_K$ extends to a $G$-equivariant morphism $f$ which fits into the commutative diagram
\[
\begin{tikzcd}
G_{S'} \arrow[r, "\varphi_{S'}"] \arrow[d, "\pi_a"'] & H_{S'} \arrow[d, "\pi_b"] \\
X \arrow[r, "f"] & Y
\end{tikzcd}
\]
where $\pi_a,\pi_b$ are the finite flat morphisms as in the definition of maximal model, constructed from the sections $a$ and $b$.\footnote{This means that the associated morphism $G_{S'}\to X_{S'}$ (resp. $H_{S'}\to Y_{S'}$) maps the unit section of $G_{S'}$ to the section $a\in X(S')$ (resp. $b\in Y(S')$).} Moreover, if $\varphi$ is faithfully flat, then so is $f$.
\end{theo*}

\begin{proof}
Indeed, the commutative diagram for generic fibers already exists
\[
\begin{tikzcd}
G_{S'}\otimes K \arrow[r, "\varphi_{S'}\otimes K"] \arrow[d, "\pi_a\otimes K"'] & H_{S'}\otimes K \arrow[d, "\pi_b\otimes K"] \\
X_K \arrow[r, "f_K"] & Y_K
\end{tikzcd}
\]
where $f_K\otimes K'$ sends $a_{K'}$ to $b_{K'}$, and we want to extend it to $S$. First we claim that the morphism $(\varphi_{S'}, \varphi_{S'})$ induces the following dashed arrow as a $G$-equivariant morphism
\[
\begin{tikzcd}
G_{S'}\times_S G_{S'} \arrow[r] & H_{S'}\times_S H_{S'} \\
G_{S'}\times_X G_{S'} \arrow[u, hook] \arrow[r, dashed, "\Psi"] & H_{S'}\times_Y H_{S'} \arrow[u, hook]
\end{tikzcd}
\]
The arrow $\Psi$ already exists over $K$, the only thing to check is that every image of $\Psi_K$ extends to $S$. Indeed, we may check this condition by passing to $S'$, and by $G$-equivariance we only need to have one image of $\Psi_{K'}$ that extends to $S'$. This is guaranteed by the extension property, provided by the extension of $f_{K'}(a_{K'})$ to $b$ in $Y(S')$. \\

From the claim, we obtain a morphism of groupoids
\[
\begin{tikzcd}
G_{S'}\times_X G_{S'} \arrow[r, shift left] \arrow[r, shift right] \arrow[d, "\Psi"] & G_{S'} \arrow[d] \\
H_{S'}\times_Y H_{S'} \arrow[r, shift left] \arrow[r, shift right] & H_{S'}
\end{tikzcd}
\]
hence it induces a $G$-equivariant morphism $f:X\to Y$ which extends $f_K$. \\

If $\varphi$ is faithfully flat, we apply the fiberwise criterion for flatness. By restricting $f:X\to Y$ to fibers, we obtain fibers of homomorphisms $\varphi:G\to H$, which are all flat. Since $\varphi_{S'},\pi_X,\pi_Y$ are all faithful, it implies that $f$ is also faithful.
\end{proof}

\begin{rema*}
The condition (E) holds, for example, if the group schemes are proper smooth or finite flat commutative. Since in these cases, maximal models exist and they are proper, cf. Prop \ref{prop1}, Prop \ref{prop2}. \\
\end{rema*}

Another nice property of maximal model is the compatibility with inductions. Let $\varphi:G\to H$ be a homomorphism of flat $S$-group schemes of finite type, and $X_K$ is a $G_K$-torsor over ${\rm Spec}(K)$ whose maximal model exists and is denoted by $X$. Let $Y_K$ be the induced $H_K$-torsor\footnote{Here the induced torsor is a priori only an algebraic space. In the results concerning induction, \textit{torsors} are allowed to be algebraic spaces. The theory of maximal model extends to the setting of algebraic spaces without essential difficulties.}
\[
Y_K := {\rm Ind}_{G_K}^{H_K}X_K = X_K\times^{G_K}H_K,
\]
and moreover we assume that its maximal model $Y$ exists. In case that $X$ is an actual $G$-torsor, we know that $Y\simeq X\times^G H$ is the induced $H$-torsor. In general, this isomorphism remains true. \\

\begin{theo*}\label{thm:induction}
Let $\varphi:G\to H$ be a homomorphism of flat $S$-group schemes of finite type, $X_K$ is a $G_K$-torsor. Assume that the maximal model $X$ of $X_K$ exists. If the induced $H_K$-torsor $Y_K$ has a maximal model $Y$, then ${\rm Ind}_G^H X = X\times^G H$ is representable by $Y$.
\end{theo*}

\begin{proof}
We only need to verify the universal property of induction for $Y$. Let $Z$ be an $S$-scheme acting by $H$, with an $G$-equivariant morphism from $X$
\[
\begin{tikzcd}[column sep=tiny]
X \arrow[rr] \arrow[rd] & & Z \\
& Y \arrow[ru, dashed]
\end{tikzcd}
\]
We want to construct the dashed arrow as an $H$-equivariant morphism from $Y$, and to show that it is unique. The dashed arrow is already uniquely defined on the generic fiber, namely, we have a unique factorization by an $H_K$-equivariant morphism from $Y_K$
\[
\begin{tikzcd}[column sep=tiny]
X_K \arrow[rr] \arrow[rd] & & Z_K \\
& Y_K \arrow[ru]
\end{tikzcd}
\]
In particular, if there exist two $H$-equivariant morphisms $Y\to Z$ which extend $Y_K\to Z_K$, then they must coincide. \\

Let $\Gamma_K$ be the graph of the morphism $Y_K\to Z_K$, and let $\Gamma$ be the schematic closure of $\Gamma_K$ in $Y\times Z$. The $H_K$-action on $\Gamma_K$ naturally extends to an $H$-action on $\Gamma$. Since $\Gamma_K$ is isomorphic to $Y_K$, the schematic closure $\Gamma$ is an integral model of $Y_K$. By maximality of $Y$, the model morphism $\Gamma\to Y$ is necessarily an isomorphism, and hence we obtain the unique factorization
\[
\begin{tikzcd}[column sep=tiny]
X \arrow[rr] \arrow[rd] & & Z \\
& Y\simeq\Gamma \arrow[ru]
\end{tikzcd}
\]
\end{proof}

\begin{coro*}\label{coro:induction}
Let $\varphi: G\to H$ be a faithfully flat homomorphism of flat $S$-group schemes of finite type with $N=\ker(\varphi)$. Let $X_K$ be a $G_K$-torsor and assume that its maximal model $X$ exists, and let $Y_K$ be the induced $H_K$-torsor ${\rm Ind}_{G_K}^{H_K}X_K$. If the maximal model $Y$ of $Y_K$ exists, then the induction ${\rm Ind}_G^H X$ and the quotient $X/N$ are representable by $Y$.
\end{coro*}

\begin{proof}
Representability of ${\rm Ind}_G^H X$ is by Theorem \ref{thm:induction}. We claim that $X/N$ is representably by ${\rm Ind}_G^H X$. To show this, We verify the universal property of quotient for ${\rm Ind}_G^H X$. Let $F:X\to Z$ be an $N$-invariant map. Consider the map $\tilde{F}:G\times X\to Z$ by sending $(g,x)$ to $f(gx)$. With the $N$-action on $G\times X$ via left multiplication on $G$, and by the $N$-invariance of $F$, the map $\tilde{F}$ factors through $H\times X$
\[
\begin{tikzcd}[column sep=tiny]
G\times X \arrow[rr, "\tilde{F}"] \arrow[rd, "/N"'] & & Z \\
& H\times X \arrow[ru, "\tilde{F}'"']
\end{tikzcd}
\]
Since $\tilde{F}'$ is $G$-invariant, it induces a map
\[
{\rm Ind}_G^H X = (H\times X)/G \longrightarrow Z
\]
and we obtain a factorization
\[
\begin{tikzcd}[column sep=tiny]
X \arrow[rr] \arrow[rd] & & Z \\
& {\rm Ind}_G^H X \arrow[ru]
\end{tikzcd}
\]
The uniqueness of this factorization follows from the uniqueness on the generic fiber. Therefore ${\rm Ind}_G^H X=Y$ represents the quotient $X/N$.
\end{proof}

\section{Existence of maximal models}\label{existence}

In this section, we review the results of Lewin-M\'en\'egaux \cite{LM83} on the existence of the maximal model of a torsor over ${\rm Spec}(K)$ under proper smooth group schemes and finite flat commutative group schemes, and we provide full details. \\

\subsection{The case of proper smooth group schemes}

Note that by Stein factorization, a proper smooth $S$-group scheme is an extension of a finite \'etale $S$-group scheme by an abelian scheme. Since Lewin-M\'en\'egeaux's proof does not depend on geometric connectivity, hence it immediately extends to proper smooth case. Let $G$ be a proper smooth $S$-group scheme. Suppose that we have a $G_K$-torsor $X_K$ over $K$, trivialized by a finite extension $K'/K$ of fields, i.e., we have the diagram
\[
\begin{tikzcd}
X_{K'} \arrow[r, "\sim"] \arrow[d] & G_{K'} \arrow[ld, "v"] \\
X_K
\end{tikzcd}
\]
we then obtain a finite flat $G_K$-equivariant $K$-morphism $v$. \\

Let $S'$ be the normalization of $S$ in $K'$. Consider the graph $Z=G_{K'}\times_{X_K} G_{K'}\subset G_{K'}\times_{K} G_{K'}$ of the equivalence relation defined by $v$, and let $\overline{Z}$ be the schematic closure of $Z$ in $G_{S'}\times_{S} G_{S'}$.

\begin{lemma}
The projection $p_1:\overline{Z}\to G_{S'}$ is finite flat.
\end{lemma}

\begin{proof}
Since $v:G_{K'}\to X_K$ is $G_K$-equivariant, the schematic closure $\overline{Z}$ is stable by the diagonal $G$-action on $G_{S'}\times_S G_{S'}$. Hence the projection $p_1:\overline{Z}\to G_{S'}$ is also $G$-equivariant. \\

We show that $p_1$ is surjective. Indeed, since $p_1$ is proper, and its image contains the generic fiber of $G_{S'}$ which is dense, hence it is surjective. Next we show that $p_1$ is finite. It suffices to prove that $p_1$ is quasi-finite, then finiteness follows from properness of $\overline{Z}$. Notice that we have
\[
\dim{\overline{Z}_k} = \dim{\overline{Z}_K} = \dim{G_{K'}} = \dim{G_{S'}\otimes k}
\]
since $\overline{Z}$ and $G_{S'}$ are both flat over $S$. Thus over an open dense subscheme of $G_{S'}$, $p_1$ is quasi-finite. By the $G$-action and the $G$-equivariance of $p_1$, it is therefore quasi-finite. \\

Finally we show the flatness of $p_1$. Let $\eta'$ be a generic point of the special fiber of $G_{S'}$, then the local ring $\mathcal{O}=\mathscr{O}_{G_{S'},\eta'}$ is a discrete valuation ring. Notice that here we use the smooth condition of $G$, hence it satisfies Serre's $R1$ condition. We have the cartesian squares
\[
\begin{tikzcd}
\overline{Z}_{\mathcal{O}} \arrow[r, hook] \arrow[d] & G_{S'}\otimes\mathcal{O} \arrow[r] \arrow[d] & {\rm Spec}(\mathcal{O}) \arrow[d, "\text{flat}"] \\
\overline{Z} \arrow[r, hook] & G_{S'}\times_S G_{S'} \arrow[r] & G_{S'}
\end{tikzcd}
\]
where $\overline{Z}_{\mathcal{O}}$ is the schematic closure of $\overline{Z}_K\otimes\text{Frac}(\mathcal{O})$, because the formation of schematic closure commutes with flat base change $\text{Spec}(\mathcal{O})/S'$ of discrete valuation rings. Thus $\overline{Z}_{\mathcal{O}}$ is flat over $\mathcal{O}$, the morphism $p_1$ is flat over an open subscheme. By $G$-equivariance of $p_1$, it is therefore flat.
\end{proof}

\begin{lemma}
Let $Y$ be a flat $S$-scheme, and $\Gamma\subset Y\times_S Y$. Suppose that
\begin{enumerate}
  \item[(1)] $\Gamma$ is the schematic closure of $\Gamma_K\subset Y_K\times_K Y_K$ and $\Gamma_K$ defines a flat equivalence relation over $Y_K$;
  \item[(2)] The two projections $\Gamma\rightrightarrows Y$ are flat.
\end{enumerate}
Then $\Gamma$ defines a flat equivalence relation over $Y$.
\end{lemma}

\begin{proof}
To show that $\Gamma$ defines a flat equivalence relation on $Y$, we need three morphisms:
\begin{enumerate}
    \item[1.] Reflexivity $e:Y\to \Gamma$;
    \item[2.] Symmetry $i:\Gamma\to\Gamma$;
    \item[3.] Transitivity $\Gamma\times_{p_1,Y,p_2}\Gamma\to\Gamma$, where $p_1,p_2$ are two projections $\Gamma\rightrightarrows Y$.
\end{enumerate}
For reflexivity, the morphism $e_K$ is induced by the diagonal
\[
\begin{tikzcd}
Y \arrow[r, dashed, "e"] \arrow[rr, bend left=15, "\Delta"] & \Gamma \arrow[r, hook] & Y\times Y \\
Y_K \arrow[u, hook] \arrow[r, "e_K"] \arrow[rr, bend right=15, "\Delta_K"'] & \Gamma_K \arrow[u, hook] \arrow[r, hook] & Y_K\times Y_K \arrow[u, hook]
\end{tikzcd}
\]
by taking schematic closures, we see that the image $\Delta(Y)$ lies in $\Gamma$, hence it induces the reflexivity morphism $e:Y\to\Gamma$. The morphisms of symmetry and transitivity are obtained similarly.
\end{proof}

\begin{prop}\label{prop1}
Let $G$ be a proper smooth group scheme over $S$, and $X_K$ is a $G_K$-torsor over $K$. Then there exists a maximal $S$-model $X$ of the torsor $X_K$, it is proper over $S$ and regular.
\end{prop}

\begin{proof}
The construction of a maximal $S$-model $X$ of $X_K$ is by taking the groupoid quotient
\[
\overline{Z} \rightrightarrows G_{S'} \rightarrow X
\]
where the right arrow is the required finite flat $G$-equivariant $S$-morphism in the definition of the maximal model. Indeed, by \cite{SGA3-1} Expos\'e V, Th\'eor\`eme 4.1, $X$ is representable by a scheme. The properness of $X$ follows from faithfully flat descent, and the regularity follows from \cite{EGA} IV, Proposition 17.3.3.
\end{proof}

\subsection{The case of finite flat commutative group schemes}

In this section, let $G$ be a finite flat commutative $S$-group scheme. It is well-known that $G$ can be embedded into an abelian $S$-scheme $A$ (cf. \cite{BBM82} Th\'eor\`eme 3.1.1), and one has an exact sequence
\[
\begin{tikzcd}
0 \arrow[r] & G \arrow[r] & A \arrow[r] & B \arrow[r] & 0
\end{tikzcd}
\]
where $B:=A/G$ is also an abelian scheme. Given a $G$-torsor $P$, one can form the induced $A$-torsor $(P\times A)/G$,\footnote{Here the induced $A$-torsor $(P\times A)/G$ is indeed representable by a scheme, by \cite{SGA3-1} Expos\'e V, Th\'eor\`eme 4.1.} which induces a trivial $B$-torsor. Conversely, given an $A$-torsor $P'$ which induces a trivial $B$-torsor, the preimage of the unit section of
\[
P' \longrightarrow (P'\times B)/A = B
\]
gives a $G$-torsor. The processes are mutually inverse. \\

\begin{prop}\label{prop2}
Let $X_K$ be a $G_K$-torsor over $K$. Then there exists a maximal model $X$ of $X_K$, which is finite flat over $S$. Moreover, $X$ is a complete intersection over $S$.
\end{prop}

\begin{proof}
Let $Y_K$ be the $A_K$-torsor induced by the $G_K$-torsor $X_K$. By Proposition \ref{prop1}, there is a maximal model $Y$ of $Y_K$. Let $X$ be the schematic closure of $X_K$ in $Y$. The $A$-action on $Y$ induces the extended $G$-action on $X$. We claim that $X$ is the maximal model of the $G_K$-torsor $X_K$. \\

Let $K'/K$ be a field extension which trivializes the $A_K$-torsor $Y_K$, and $S'$ the normalization of $S$ in $K'$. From the commutative diagram
\[
\begin{tikzcd}
G_{S'} \arrow[d, dashed] \arrow[r, hook] & A_{S'} \arrow[r] \arrow[d] & B_{S'} \arrow[d] \\
X \arrow[r, hook] & Y \arrow[r] & B
\end{tikzcd}
\]
it induces a finite flat $G$-equivariant $S$-morphism $G_{S'}\to X$. Therefore $X$ is the maximal model of $X_K$. The fact that $X/S$ is a complete intersection is indicated from that $X\subset Y$ is defined by the preimage of the unit section of $Y\to B$.
\end{proof}

\begin{rema}
From the proofs of Proposition \ref{prop1} and Proposition \ref{prop2}, for any finite extension $K'/K$ which trivializes the $G_K$-torsor $X_K$, it is always possible to construct a finite flat $G$-equivariant $S$-morphism $G_{S'}\to X$, where $S'$ is the normalization of $S$ in $K'$, and $X$ is the maximal model. Thus the requirement for $K'/K$ in the definition of the minimal model is only to trivialize $X_K$, in the case of proper smooth group schemes and finite flat commutative group schemes. \\
\end{rema}

\section{Maximal model of torsors under a semi-direct product}\label{semi-direct_prod}

Let $G = N\rtimes H$ be a flat $S$-group scheme, where $N\vartriangleleft G$ is a flat normal subgroup scheme. In this section, we study relations between maximal models under $G$ and those under $N$ and $H$. \\

\begin{theo*}\label{theorem:semi-direct_prod}
Let $X_K$ be an $G_K$-torsor over $K$, and $Y_K={\rm Ind}_{G_K}^{H_K}X_K$ the induced $H_K$-torsor. If the maximal model $Y$ of $Y_K$ exists, then the following are equivalent
\begin{enumerate}
    \item[(1)] The maximal model $X$ of $X_K$ exists;
    \item[(2)] The maximal model $X\to Y$ of the $N_{Y_K}$-torsor $X_K\to Y_K$ exists.
\end{enumerate}
Moreover, $X$ is the common maximal model in (1) and (2), if it exists.
\end{theo*}

\begin{proof}
$(1)\Rightarrow (2)$: We only need to check that the morphism $X\to Y$ is indeed the maximal model of the $N_{Y_K}$-torsor $X_K\to Y_K$. Let $S'/S$ be a finite flat base change that verifies the definition of $X$ and $Y$ being the maximal models of $X_K$ and $Y_K$ respectively. Then we have the following $N$-equivariant diagram
\[
\begin{tikzcd}
G_{S'\times S'} \arrow[rr, "\text{finite flat}"] \arrow[d, equal] & & X \\
N_{S'}\times_{S} H_{S'} \arrow[d, "\text{finite flat}"'] & & \\
N_{S'}\times_{S} Y \arrow[rr, "\sim"] & & N_Y \times_Y Y_{S'} \arrow[uu]
\end{tikzcd}
\]
where the up left equality is as schemes with $N_{S'}$-action. The right vertical map is $N_Y$-equivariantly induced by
\[
Y_{S'} \longrightarrow X_{S'}\longrightarrow X
\]
where the first arrow is induced by an inclusion $H\hookrightarrow G$. By the fiberwise criterion of flatness, the morphism
\[
N_Y\times_Y Y_{S'} \longrightarrow X
\]
is finite flat. Hence $X\to Y$ is the maximal model of $X_K\to Y_K$. \\

\noindent $(2)\Rightarrow (1)$: Let $h$ denote the morphism $Y\to S$. The $N_Y$-action on $X$ as an $Y$-scheme induces an $N$-action on $X$ as an $S$-scheme as follows
\[
\begin{tikzcd}
N \arrow[r] & h_*N_Y \arrow[r] & h_*{\rm Aut}_Y(X) \arrow[r] & {\rm Aut}_S(X)
\end{tikzcd}
\]
and clearly the generic fiber of this action is the $N_K$-action on $X_K$ induced from the normal subgroup $N_K\vartriangleleft G_K$. \\

Let $S'/S$ be a finite flat base change that verifies the definition of $Y$ and $X\to Y$ being the maximal models of $Y_K$ and $X_K\to Y_K$ respectively. Then we have a finite flat morphism
\[
N\times_S Y_{S'} = N_Y\times_Y Y_{S'} \longrightarrow X
\]
We claim that the $H$-action on the left (acting on $N$ and $Y$) descends to $X$. Let $\Theta$ denote $(N\times_S Y_{S'})\times_X (N\times_S Y_{S'})$. The two compositions $H\times_S\Theta\rightrightarrows X$ from the following diagram
\[
\begin{tikzcd}
H\times_S \Theta \arrow[r, shift left] \arrow[r, shift right] & H\times_S(N\times_S Y_{S'}) \arrow[d, "\text{$H$-action}"] \arrow[r] & H\times_S X \arrow[d, dashed] \\
& N\times_S Y_{S'} \arrow[r] & X
\end{tikzcd}
\]
coincide on their generic fibers, hence they agree. Since the top row is a flat groupoid because of flatness of $H/S$, it induces the dashed arrow, which is the descent $H$-action on $X$. Consequently, we obtain the $G$-action on $X$, whose generic fiber coincides with the $G_K$-action on $X_K$. \\

Finally we need to verify that $X$ is indeed the maximal model of $X$. The following diagram
\[
\begin{tikzcd}
G_{S'\times S'} \arrow[rr] \arrow[d, equal] & & X \\
N_{S'}\rtimes_{S} H_{S'} \arrow[d] & & \\
N_{S'}\times_{S} Y \arrow[rr, "\sim"] & & N_Y \times_Y Y_{S'} \arrow[uu]
\end{tikzcd}
\]
is $G$-equivariant, and the top horizontal map is finite flat.
\end{proof}

\begin{rema*}
It is crucial that $G$ is a semi-direct product, rather than a general extension of flat $S$-group schemes. If the short exact sequence
\[
\begin{tikzcd}
0 \arrow[r] & N \arrow[r] & G \arrow[r] & H \arrow[r] & 0
\end{tikzcd}
\]
does not split, then the right vertical map in the above diagram could not be defined. In other words, in general the $N_H$-torsor $G\to H$ is not the maximal model of its generic fiber, unless the sequence splits after a finite extension $S'/S$ which is the normalization of a finite extension $K'/K$ of local fields. \\
\end{rema*}

\section{The ideal sheaf of different and transitivity formula}\label{different}

In this section, let $G/S$ be either a proper smooth scheme or a finite flat commutative group scheme, $X_K$ an $G_K$-torsor, and $X$ is the maximal model of $X_K$. Let $h:G\to S$ and $\pi:X\to S$ denote the structure morphisms. In the two cases, $X$ is either regular or a local complete intersection, where for either case one can define the dualizing sheaf of $X$. Let $\omega_{G/S},\omega_{X/S}$ be the dualizing sheaves of $G/S$ and $X/S$, and $\omega$ an invertible sheaf on $S$ such that $h^*\omega=\omega_{G/S}$. \\

Let us summarize all the necessary notations in the following commutative diagram
\[
\begin{tikzcd}
S \arrow[r, bend right, "1"'] & \arrow[l, "h"'] G & \\
& G\times_S X \arrow[ld, "\sigma"] \arrow[d, "\lambda"] \arrow[rd, "q_2"] \arrow[u, "q_1"'] & \\
X \arrow[ur, bend left, "\epsilon"] \arrow[uu, "\pi"] & X\times_S X \arrow[l, "p_1"] \arrow[r, "p_2"'] & X
\end{tikzcd}
\]
where
\begin{itemize}
  \item $p_1,p_2$ are projections from $X\times_S X$,
  \item $q_1,q_2$ are projections from $G\times_S X$,
  \item $\epsilon$ is induced from the unit section of $G$,
  \item $\sigma$ is the morphism of $G$-action,
  \item $\lambda$ is $(\sigma,q_2)$.
\end{itemize}
There is a trace map from the composition $p_2\circ\lambda=q_2$
\[
{\rm\bf Tr}_{\lambda}: \lambda_*\omega_{G\times_S X/X} \longrightarrow \omega_{X\times_S X/X} \simeq p_1^*\omega_{X/S}
\]
where the structure morphisms of $G\times_S X/X$ and $X\times_S X/X$ are $q_2$ and $p_2$ respectively. Note that
\[
\omega_{G\times_S X/X} = q_1^*\omega_{G/S} = q_1^*h^*\omega = q_2^*\pi^*\omega = \lambda^*p_2^*\pi^*\omega
\]
hence the trace ${\rm\bf Tr}_{\lambda}$ and the adjunction $1\to\lambda_*\lambda^*$ induce
\[
\begin{tikzcd}
p_2^*\pi^*\omega \arrow[r, "\alpha"] \arrow[d] & p_1^*\omega_{X/S} \\
\lambda_*\lambda^*p_2^*\pi^*\omega \arrow[ur, "{\rm\bf Tr}_{\lambda}"'] &
\end{tikzcd}
\]
Let $\varphi = (\lambda\epsilon)^*\alpha: \pi^*\omega \to \omega_{X/S}$, which is an isomorphism on the generic fiber. The \textit{different ideal sheaf} of $X/S$ is defined by $\delta_{X/S}:=\pi^*\omega\otimes_{\mathscr{O}_X}\omega_{X/S}^{-1}$.\footnote{Though a priori $\delta_{X/S}$ is only a sheaf of modules, it is the image of the injection
\[
\varphi\otimes {\rm id}: \pi^*\omega\otimes_{\mathscr{O}_X}\omega_{X/S}^{-1} \longrightarrow \omega_{X/S}\otimes \omega_{X/S}^{-1}\simeq \mathscr{O}_X
\]
hence a genuine sheaf of ideals.} It turns out that this ideal sheaf of different measures how far the maximal model $X$ is from being a torsor under $G$. It plays a similar role as the usual different in the classical ramification theory of local fields. \\

\begin{prop*}
$\delta_{X/S}\simeq\mathscr{O}_X$ if and only if $X$ is a $G$-torsor over $S$.
\end{prop*}

\begin{proof}
The ``if'' part is straightforward, since $\lambda$ is an isomorphism and thus ${\rm\bf Tr}_{\lambda}$ is an isomorphism as well. \\

\noindent If $\varphi$ is an isomorphism, then $\alpha$ is an isomorphism along the diagonal. Notice that $\alpha$ is equivariant with respect to the $G$-action on the first factors of $G\times_S X$ and $X\times_S X$, therefore $\alpha$ is an isomorphism. It implies that the trace homomorphism ${\rm\bf Tr}_{\lambda}$ is surjective. Then the proposition follows from the next lemma, applying to $T=X$, $Y=X\times_S X$ and $Y'=G\times_S X$.
\end{proof}

\begin{lemma*}
Let $T$ be a flat $S$-scheme, and $\lambda:Y'\to Y$ a finite morphism of $T$-schemes, and we assume that $Y/T$ and $Y'/T$ have locally free dualizing sheaves $\omega_{Y/T}$, $\omega_{Y'/T}$. Suppose that $\lambda$ is an isomorphism over an open schematically dense subset of $Y$. If ${\rm\bf Tr}_{\lambda}:\lambda_*\omega_{Y'/T}\to\omega_{Y/T}$ is surjective, then $\lambda$ is an isomorphism.
\end{lemma*}

\begin{proof}
It suffices to show the following: Let $A\to B$ be a homomorphism of flat $R$-algebras, such that $B$ is finite as an $A$-module, and $A_K\to B_K$ is an isomorphism. If the trace map
\[
{\rm\bf Tr}: {\rm Hom}_A(B, A) \longrightarrow A
\]
is surjective, then $A\to B$ is an isomorphism. \\

Indeed, if ${\rm\bf Tr}$ is surjective, it means that there is an $A$-homomorphism $u:B\to A$ such that $u(1)=1$. Note that we have $u_K(1)=1$ and hence $u_K:B_k\to A_K$ is an isomorphism. By flatness, we have $B\subset B_K$. From the diagram
\[
\begin{tikzcd}
B \arrow[r, "u"] \arrow[d, hook] & A \arrow[d, hook] \\
B_K \arrow[r, "u_K", "\sim"'] & A_K
\end{tikzcd}
\]
we see that $u$ is injective, hence it is bijective and the inverse of $A\to B$.
\end{proof}

The next proposition is the transitivity formula for the different $\delta_{X/S}$.\\

\begin{prop*}
Let
\[
\begin{tikzcd}
1 \arrow[r] & G' \arrow[r] & G \arrow[r, "\beta"] & G'' \arrow[r] & 1
\end{tikzcd}
\]
be an exact sequence of $S$-group schemes (proper smooth schemes or finite flat commutative group schemes). Let $X_K\to X''_K$ be a morphism of torsors under $G_K$ and $G''_K$ (compatible with $\beta_K$), and $g:X\to X''$ the extended $S$-morphism of their maximal models. Then
\begin{enumerate}
  \item[(1)] $X$ is the maximal $X''$-model of the $X''_K$-torsor $X_K$ under the group $G'_K\times_K X''_K$;
  \item[(2)] $\delta_{X/S}\simeq\delta_{X/X''}\otimes g^*\delta_{X''/S}$.
\end{enumerate}
\end{prop*}

\begin{proof}
Firstly, it is clear that $X_K\to X''_K$ is a torsor under the $X''_K$-group scheme $G'_K\times_K X''_K$. Let $K'/K$ be a finite extension of fields which trivializes the torsors $X_K$ and $X''_K$, and $S'$ the normalization of $S$ in $K'$. Then one has the following diagram
\[
\begin{tikzcd}
1 \arrow[r] & G'_{S'} \arrow[r] & G_{S'} \arrow[r] \arrow[d] & G''_{S'} \arrow[r] \arrow[d] & 1 \\
& & X \arrow[r, "g"] & X'' &
\end{tikzcd}
\]
the composite morphism $G'_{S'}\to X$ gives the finite flat $(G'\times_S X'')$-equivariant $X''$-morphism
\[
(G'\times_S X'')_{S'}=G'_{S'}\times_{S'} X''_{S'} \longrightarrow X
\]
which proves (1). The transitivity formula (2) follows from the transitivity for dualizing sheaves.
\end{proof}

\section{Examples of maximal models under finite flat group schemes of order $p$} \label{sec:max_example}

Throughout this section, $R$ is a complete discrete valuation ring with the perfect residue field $k$ of characteristic $p>0$, and the fraction field $K$. Let $\pi\in R$ be a uniformizer. Moreover, we assume that $K$ contains a $p$-th root of unity whenever ${\rm char}(K)\neq p$. In the cases of $\mu_p$ and $\alpha_p$, we assume that the characteristic of $K$ is $p$. We will use the following lemma during this section.

\begin{lemma}\label{lemma:extend_normalization}
Let $G/S$ be a finite flat group scheme, and $P_K={\rm Spec}(K')$ a $G_K$-torsor over $K$, where $K'/K$ is a finite extension of fields. Let $P$ denote the normalization of $S$ in $K'$. If the $G_K$-action on $P_K$ extends to a $G$-action on $P$, then $P/S$ is the maximal model of $P_K/K$.
\end{lemma}

\begin{proof}
Indeed, the torsor $P_K$ is tautologically trivialized by itself $K'/K$. We have the following diagram
\[
\begin{tikzcd}
G_{P} \arrow[rrd, bend left, "m"] \arrow[rdd, bend right, "{\rm pr}_2"'] \arrow[rd] & & \\
& P_P \arrow[d] \arrow[r] & P \arrow[d] \\
& P \arrow[r] & S
\end{tikzcd}
\]
where $m$ is the $G$-action morphism. In order that $P/S$ is the maximal model, we need to show that $m$ is finite flat. The morphism $m$ decomposes as
\[
\begin{tikzcd}
G\times P \arrow[rr, "{\rm pr}_1\times m", "\sim"'] & & G\times P \arrow[r, "{\rm pr}_2"] & P
\end{tikzcd}
\]
hence $m$ is finite flat by the fact that ${\rm pr}_2$ is finite flat.
\end{proof}

\subsection{The \'etale-local group scheme $(\mathbb{Z}/{p\mathbb{Z}})_R$}

A torsor over $K$ under $\mathbb{Z}/{p\mathbb{Z}}$ is described by Kummer theory if ${\rm char}(K)\neq p$, by Artin--Schreier theory if ${\rm char}(K)=p$. Such a torsor is either a trivial one, or has the form ${\rm Spec}(L)$ where $L/K$ is a cyclic Galois extension of order $p$. \\

The maximal model of a trivial torsor is the trivial torsor over $R$. In the nontrivial case, let $R'$ be the integral closure of $R$ in $L$. Since $\mathbb{Z}/{p\mathbb{Z}}$-action extends to the normalization ${\rm Spec}(R')$, it is therefore the maximal model by Lemma \ref{lemma:extend_normalization}.\\

\subsection{The local-\'etale group scheme $\mu_{p,R}$}

By Kummer theory, a $\mu_{p,K}$-torsor over $K$ has the following form
\[
P_f = {\rm Spec}(K_f) := {\rm Spec}\big(K[X]/(X^p - f)\big)
\]
where $f\in K^{\times}$, and we assume $f\notin (K^{\times})^p$, which is equivalent to $P_f$ being nontrivial. The $\mu_{p,K}$-action is given by
\begin{eqnarray*}
\mu_{p,K}\times P_f & \longrightarrow & P_f \\
(z, x) & \longmapsto & z\cdot x
\end{eqnarray*}
We write $f=u\pi^i$ for $u\in R^{\times}$ and $0\leqslant i\leqslant p-1$. The situation separates into two cases: \\

\noindent \textbf{Case I}. $i=0$. In this case, $P_f$ naturally extends to a $\mu_{p,R}$-torsor
\[
\widetilde{P}_f := {\rm Spec}\big(R[X]/(X^p - f)\big)
\]
which is trivialized by the normalization of $R$ in $K_f$. Hence it is the maximal model of $P_f$, cf. Remark \ref{rmk:max_model} (3). Note that only in this case, there could be maximal models which are not regular. If $u$ modulo $\pi$ is not a $p$-power, then $R[X]/(X^p-u)$ is a discrete valuation ring, and the maximal model is regular. If $u$ is a $p$-power modulo $\pi$, let us write
\[
u = \alpha^p + \pi^r\beta
\]
for a maximal $r\geqslant 1$, where $\alpha,\beta\in R^{\times}$. The existence of such maximal $r$ is clear, since otherwise $u$ would be a $p$-power element in $R$. Let $Y=X-\alpha$, we see
\[
\frac{R[X]}{(X^p - u)} = \frac{R[Y]}{(Y^p - \pi^r\beta)}
\]
thus the maximal model is not regular if $r>1$. \\

\noindent \textbf{Case II}. $i>0$. In this case, we will see that the maximal model is always the normalization. Let $m,n\in\mathbb{N}$ with $mi-np=1$, and $\tau:=\pi^{-n}X^m\in K_f$. Then
\[
R[\tau] = R[T]/(T^p-u^m\pi)
\]
is a discrete valuation ring and it is the normalization of $R$ in $K_f$. The $\mu_p$-action naturally extends to ${\rm Spec}(R[\tau])$ by
\begin{eqnarray*}
\mu_{p,R}\times {\rm Spec}(R[\tau]) & \longrightarrow & {\rm Spec}(R[\tau]) \\
(z, \tau) & \longmapsto & z^m\cdot\tau
\end{eqnarray*}
thus ${\rm Spec}(R[\tau])$ is the maximal model of $P_f$ by Lemma \ref{lemma:extend_normalization}. \\

\subsection{The local-local group scheme $\alpha_{p,R}$}

From the short exact sequence
\[
\begin{tikzcd}
0 \arrow[r] & \alpha_{p,K} \arrow[r] & \mathbb{G}_{a,K} \arrow[r] & \mathbb{G}_{a,K} \arrow[r] & 0
\end{tikzcd}
\]
we know that $H^1(K,\alpha_p)\simeq K/K^p$, so an $\alpha_p$-torsor over $K$ has the form
\[
P_f = {\rm Spec}(K_f) := {\rm Spec}\big(K[X]/(X^p - f)\big)
\]
where $f\in K$ and we assume $f\notin K^p$. The $\alpha_{p,K}$-action is given by
\begin{eqnarray*}
\alpha_{p,K}\times P_f & \longrightarrow & P_f \\
(a, x) & \longmapsto & x+a
\end{eqnarray*}
We write $f = u\pi^i$ for $u\in R^{\times}$ and $|i|\geqslant 1$, this is always possible by choosing an appropriate representative of $[f]\in K/K^p$, which represents the same isomorphism class of $P_f$. The situation also separates into two cases: \\

\noindent \textbf{Case I}. $i>0$. In this case, $P_f$ naturally extends to an $\alpha_{p,R}$-torsor
\[
\widetilde{P}_f := {\rm Spec}\big(R[X]/(X^p-f)\big)
\]
and it is the maximal model of $P_f$. Moreover, $\widetilde{P}_f$ is regular if and only if $i=1$. \\

\noindent \textbf{Case II}. $i<0$. First, we change the coordinate by $Y=X^{-1}$,
\[
P_f = {\rm Spec}\big(K[Y]/(Y^p-u^{-1}\pi^{-i})\big)
\]
and the $\alpha_{p,K}$-action goes by
\begin{eqnarray*}
\alpha_{p,K}\times P_f & \longrightarrow & P_f \\
(a, y) & \longmapsto & \frac{y}{1+ay}
\end{eqnarray*}
Let $m,n\in\mathbb{N}$ with $m\cdot(-i) - np = 1$, and let $\tau = \pi^{-n}Y^m$. Then
\[
R[\tau] = R[T]/(T^p - u^{-m}\pi)
\]
is the normalization of $R$ in $K_f$. The $\alpha_{p,K}$-action extends to $\widetilde{P}_f := {\rm Spec}(R[\tau])$ by
\begin{eqnarray*}
\alpha_{p,R}\times \widetilde{P}_f & \longrightarrow & \widetilde{P}_f \\
(a, \tau) & \longmapsto & \frac{\tau}{(1+au^n\tau^{-i})^m}
\end{eqnarray*}
therefore $\widetilde{P}_f$ is the maximal model of $P_f$ by Lemma \ref{lemma:extend_normalization}, and it is always regular. \\

\subsection{The congruence group schemes $H_{\lambda}$}

Let $\lambda\in R$ be an element satisfying
\[
v_R(\lambda) \leqslant \frac{v_R(p)}{p-1}
\]
in particular, the condition is void if $R$ has characteristic $p$. The \textit{congruence $R$-group scheme $H_{\lambda}$ of level $\lambda$} has the underlying scheme structure
\[
H_{\lambda} = {\rm Spec} \frac{R[x]}{\big((1+\lambda x)^p - 1\big)/\lambda^p}
\]
and the group law is given by
\[
x_1\circ x_2 = x_1 + x_2 + \lambda\cdot x_1x_2.
\]
Let $\mu:=p/\lambda^{p-1}$ (if $\lambda=0$, then let $\mu\in R$), this is an element of $R$ by looking at its valuation
\[
v_R(\mu) = v_R(p) - (p-1)v_R(\lambda) \geqslant 0.
\]
Then the group scheme $H_{\lambda}$ fits into a Kummer-type sequence (cf. \cite{AG07}, Appendix A)
\[
\begin{tikzcd}
0 \arrow[r] & H_{\lambda} \arrow[r] & \mathcal{G}^{\lambda} \arrow[r, "\varphi^{\lambda}"] & \mathcal{G}^{\lambda^p} \arrow[r] & 0
\end{tikzcd}
\]
where $\mathcal{G}^{\lambda}$ is an affine smooth one-dimensional $R$-group scheme with the underlying scheme structure
\[
\mathcal{G}^{\lambda} = {\rm Spec}\bigg(R\bigg[x,\frac{1}{1+\lambda x}\bigg]\bigg)
\]
and the group law
\[
x_1\circ x_2 = x_1 + x_2 + \lambda\cdot x_1x_2.
\]
The isogeny $\varphi^{\lambda}$ is defined explicitly by
\[
\varphi^{\lambda}(x) := \big((1+\lambda x)^p-1\big)/\lambda^p = x^p + \sum_{i=1}^{p-1}\frac{1}{p}\binom{p}{i}\lambda^{i-1}\mu x^i.
\]

We have the following result on structure of the special fiber $(H_{\lambda})_k$:

\begin{lemma}{\rm [\cite{AG07} Lemma 2.2]}
\[
(H_{\lambda})_k = \left\{
\begin{array}{cl}
\mu_{p,k} & \text{if}\ \ v_R(\lambda)=0; \\
\alpha_{p,k} & \text{if}\ \ v_R(p)=\infty,\ v_R(\lambda)>0; \\
\alpha_{p,k} & \text{if}\ \ \infty > v_R(p) > (p-1)v_R(\lambda)>0; \\
\mathbb{Z}/p\mathbb{Z} & \text{if}\ \ \infty > v_R(p) = (p-1)v_R(\lambda).
\end{array}
\right.
\]
\end{lemma}
In the following, we will calculate the maximal model under assumptions $0<v_R(\lambda)<\infty$ and $v_R(p)=\infty$. \\

Under our assumptions, we have $\mu=0$, hence $\varphi^{\lambda}(x)=x^p$, and the congruence group scheme $H_{\lambda}$ is local. An $(H_{\lambda})_K$-torsor $P_f$ has the form
\[
P_f := {\rm Spec}\big(K[W]/(W^p-f)\big)
\]
with $f\in K$. We assume that $P_f$ is nontrivial, hence $f\neq 0$. The action is given by
\begin{eqnarray*}
(H_{\lambda})_K\times P_f & \longrightarrow & P_f \\
(x, w) & \longmapsto & x+w+\lambda xw \\
\end{eqnarray*}

\begin{lemma}
If $v_R(f)\geqslant 0$, then $P_f$ extends to an $H_{\lambda}$-torsor.
\end{lemma}

\begin{proof}
If $v_R(f)\geqslant 0$, we claim that $\widetilde{P}_f={\rm Spec}\big(R[W]/(W^p-f)\big)$ is an $H_{\lambda}$-torsor. It is clear that the $(H_{\lambda})_K$-action on $P_f$ extends to an $H_{\lambda}$-action on $\widetilde{P}_f$ via the same formula. We need to show that the following morphism is an isomorphism
\begin{eqnarray*}
H_{\lambda}\times \widetilde{P}_f & \longrightarrow & \widetilde{P}_f\times \widetilde{P}_f \\
\frac{R[W_1,W_2]}{(W_1^p-f, W_2^p-f)} & \longrightarrow & \frac{R[x,W]}{(x^p, W^p-f)} \\
W_1 & \longmapsto & x+W+\lambda xW \\
W_2 & \longmapsto & W
\end{eqnarray*}
We can write an inverse morphism formally as
\begin{eqnarray*}
W & \longmapsto & W_2 \\
x & \longmapsto & \frac{W_1-W_2}{1+\lambda W_2}
\end{eqnarray*}
Notice that
\[
\frac{1}{1+\lambda W_2} = \sum_{i=1}^{p-1}\Big(\sum_{k=0}^{\infty}(-\lambda)^{i+kp}f^k\Big)W_2^i
\]
the valuation of each term in the coeafficient of $W_2^i$
\[
v_R((-\lambda)^{i+kp}f^k) = iv_R(\lambda) + k(v_R(f)+pv_R(\lambda))
\]
turns to $+\infty$ as $k\to +\infty$. Hence each coefficient converges, and this formal inverse is an actual inverse. Thus $\widetilde{P}_f$ is an $H_{\lambda}$-torsor which extends $P_f$.
\end{proof}

Consequently, in the case $v_R(f)\geqslant 0$, the extended $H_{\lambda}$-torsor $\widetilde{P}_f$ is the maximal model of $P_f$. \\

\begin{lemma}
If $v_R(f)<0$, then the $(H_{\lambda})_K$-action on $P_f$ extends to its normalization.
\end{lemma}

\begin{proof}
Write $f=\delta\pi^{-i}$ for some $\delta\in R^{\times}$ and $i\in\mathbb{N}$, and let $m,n\in\mathbb{N}$ such that $mi-np=1$. We change the coordinate $W\mapsto U^{-1}$ of $P_f$, and the $(H_{\lambda})_K$-action on $P_f$ goes like
\begin{eqnarray*}
(H_{\lambda})_K\times P_f & \longrightarrow & P_f \\
(x, u) & \longmapsto & \frac{u}{1+ux+\lambda x}
\end{eqnarray*}
As in previous lemma, let $\widetilde{P}_f$ denote the $R$-scheme ${\rm Spec}\big(R[U]/(U^p-f^{-1})\big)$. Let $\mathcal{P}_f$ be the normalization of $\widetilde{P}_f$, which is
\begin{eqnarray*}
\mathcal{P}_f & \longrightarrow & \widetilde{P}_f \\
R[U]/(U^p-\delta^{-1}\pi^i) & \longrightarrow & R[T]/(T^p-\delta^{-m}\pi) \\
U & \longmapsto & \delta^{n}T^i
\end{eqnarray*}
Then the $H_{\lambda}$-action on $\widetilde{P}_f$ extends to $\mathcal{P}_f$ as follows
\begin{eqnarray*}
H_{\lambda}\times\mathcal{P}_f & \longrightarrow & \mathcal{P}_f \\
(x, \tau) & \longmapsto & \frac{\tau}{(1+\delta^n\tau^ix+\lambda x)^m}
\end{eqnarray*}
Hence the normalization $\mathcal{P}^f$ is the maximal model of $P_f$ in this case.
\end{proof}

\begin{rema}
The two situations are not entirely disjoint, it may happen that the $R$-scheme $\widetilde{P}_f$ is an $H_{\lambda}$-torsor and it is already regular. For example, this happens if $f=\pi$. \\
\end{rema}

\section{Appendix: A result of non-flat descent}

In this appendix, we show a result on certain morphisms being effectively epimorphic. For the definition of effective epimorphism, see \cite{FGA} 212-03.\\

Let us recall the definition of pure morphism. We fix a base scheme $S={\rm Spec}(R)$, where $R$ is a discrete valuation ring, $K$ its fractional field, $k$ its residue field. We denote the henselization of $S$ by $S^h$, and correspondingly $X^h$ by $X\times_{S}S^h$. \\

\begin{defi*}
Let $X$ be a $S$-scheme locally of finite type. It is called {\it $S$-pure} if the closure of any associated point of the generic fiber of $X^h$ meets its special fiber. \\
\end{defi*}

\begin{lemma*}\label{lemma:eff_epi}
Let $u:Z'\to Z$ be an $S$-morphism between finite flat $S$-schemes. If $u$ is schematically dominant, then $u$ is an effective epimorphism.
\end{lemma*}

\begin{proof}
Let $A,A'$ be the function algebras of $Z,Z'$ respectively. Since $u$ is finite, by \cite{SGA1} Expos\'e VIII, Proposition 5.1, it is an effective epimorphism if the sequence
\[
\begin{tikzcd}
A \arrow[r] & A' \arrow[r, shift left] \arrow[r, shift right] & A'\otimes_A A'
\end{tikzcd}
\]
is exact. The first map is injective by assumption. Let $\{e_1,e_2,...,e_n\}$ be a basis for $A'$ with $e_1=1$.\footnote{This is possible, since $A'$ is a finite free module, and the sequence $0\to R\to A'\to A'/R\to 0$ splits.} By the structure theorem for modules over a principal domain, there are natural numbers $a_2\leqslant ...\leqslant a_m$ with $m\leqslant n$ such that $\{1,\pi^{a_2}e_2,...,\pi^{a_m}e_m\}$ is a basis for $A$. Moreover, $\{e_i\otimes_R e_j\}$ is a basis for $A'\otimes_R A'$, and $A'\otimes_A A'$ is a quotient of it by the submodule generated by elements of the form $\pi^{a_i}(e_i\otimes_R 1 - 1\otimes_R e_i)$ for $2\leqslant i\leqslant m$.

Let $x=x_1+\sum_{i=2}^n x_ie_i\in A'$ such that $x\otimes_A 1=1\otimes_A x$ in $A'\otimes_A A'$, then there are elements $y_1,...,y_m\in R$ with
\[
\sum_{i=2}^n x_i(e_i\otimes 1 - 1\otimes e_i) = \sum_{i=2}^m y_i\pi^{a_i}(e_i\otimes 1 - 1\otimes e_i),
\]
theerefore $x=x_1+\sum_{i=2}^m y_i\pi^{a_i}e_i$ is in $A$.
\end{proof}

The main result of this appendix is the following

\begin{prop*}\label{prop:eff_epi}
Let $u:Z'\to Z$ be a finite morphism between flat $S$-schemes of finite type, such that $u_K$ is an effective epimorphism. Then $u$ is an effective epimorphism, and it remains true after any flat finite type base change $T\to S$. \\
\end{prop*}

\begin{proof}
It suffices to prove it over the henselization $S^h$ of $S$, so we assume $S$ is henselian and keep the same notations. Let $Z'':=Z'\times_Z Z'$ and $v\:Z''\to Z$ denotes the canonical map. Since $u$ is finite, we need to show that the sequence
\[
\begin{tikzcd}
\mathscr{O}_Z \arrow[r] & u_*\mathscr{O}_{Z'} \arrow[r, shift left, "{\rm pr}_1"] \arrow[r, shift right, "{\rm pr}_2" swap] & v_*\mathscr{O}_{Z''}
\end{tikzcd}
\]
is exact. Since $Z,Z'$ are flat, and $u_K$ is schematically dominant, the first arrow is injective. \\

Let us show the exactness in the middle. The question is Zariski local on $Z$. Note that the exactness over the open $Z_K$ is clear, since $u_K$ is an effective epimorphism. It remains to show the exactness over the neighborhood of the special fiber. By \cite{Rom12} Lemma 2.1.7 and Lemma 2.1.11, for any fixed point $z\in Z$ in the special fiber, we can find an affine open neighborhood which is $S$-pure. By restricting $Z',Z''$ to such an affine open, we reduce to the case where $Z$ is affine and $S$-pure. Now, it suffices to prove that for any function $f':Z'\to\mathbb{A}^1_R$ such that ${\rm pr}_1^*f' = {\rm pr}_2^*f'$, then $f'$ descends to $f:Z\to\mathbb{A}^1_R$. By {\it op.cit.} Proposition 3.2.5, it is sufficient to check it on the generic fiber and all finite flat closed subschemes of $Z$. Indeed, since $u_K$ is an effective epimorphism, we already have the descent function on $Z_K$. For finite flat closed subschemes, by restricting $Z',Z''$ to a finite flat subscheme of $Z$, we can apply Lemma \ref{lemma:eff_epi} to descend $f'$ to $Z$. \\

Finally, after any flat finite type base change $T\to S$, the generic fiber of the morphism $u_T:Z'_T\to Z_T$ is still an effective epimorphism. Hence by applying the result in the first part, we conclude that $u_T$ is an effective epimorphism.
\end{proof}

\begin{rema*}
Recall that in the definition of maximal mode, we have a finite flat morphism $f:G_{T_{S'}}\to X$, see Remark \ref{rmk:max_model} (2). By Proposition \ref{prop:eff_epi}, the induced non-flat morphism $g:G_{T_{S'}}\to X_{T_{S'}}$ is an effective epimorphism. \\
\end{rema*}

\addcontentsline{toc}{section}{References}

\bigskip

\noindent
Yuliang HUANG,
{\sc Universit\'e de Rennes 1, CNRS, IRMAR - UMR 6625, F-35000 Rennes, France} \\
Email address: {\tt yu-liang.huang@univ-rennes1.fr, a5h\_4723@163.com}


\begin{thebibliography}{200}
\bibitem[SP19]{SP19} The Stacks Project Authors: \textit{Stacks Project}, \href{https://stacks.math.columbia.edu}{link}.
\bibitem[AG07]{AG07} Andreatta F., Gasbarri C.: \textit{Torsors under some group schemes of order $p^n$}, J. Algebra 318, No. 2, 1057--1067 (2007)
\bibitem[AS02]{AS02} Abbes A., Saito T.: \textit{Ramification of local fields with imperfect residue fields}, Am. J. Math. 124, No. 5, 879--920 (2002)
\bibitem[AS03]{AS03} Abbes A., Saito T.: \textit{Ramification of local fields with imperfect residue fields. II}, Doc. Math. Extra Vol., 5--72 (2003)
\bibitem[BBM82]{BBM82} Berthelot P., Breen L., Messing W.: \textit{Th\'eorie de Dieudonne Cristalline. II}, Lecture Notes in Mathematics 930, Springer (1982)
\bibitem[CEPT96]{CEPT96} Chinburg T., Erez B., Pappas G., Taylor M. J.: \textit{Tame actions of group schemes: Integrals and slices}, Duke Math. J. 82, No. 2, 269--308 (1996)
\bibitem[EGA]{EGA} Grothendieck A., Dieudonn{\'e} J.: \textit{{\'E}l{\'e}ments de G{\'e}om{\'e}trie Alg{\'e}brique}, Publ. Math. IH{\'E}S 4 (Chapter 0, 1-7, and I, 1-10), 8 (II, 1-8), 11 (Chapter 0, 8-13, and III, 1-5), 17 (III, 6-7), 20 (Chapter 0, 14-23, and IV, 1), 24 (IV, 2-7), 28 (IV, 8-15), and 32 (IV, 16-21), 1960-1967.
\bibitem[FGA]{FGA} Grothendieck A.: \textit{Fondements de la g\'eom\'etrie alg\'ebrique}. Extraits du S\'eminaire Bourbaki 1957-1962, Paris: Secr\'etariat math\'ematique (1962)
\bibitem[KM85]{KM85} Katz N. M., Mazur B.: \textit{Arithmetic Moduli of Elliptic Curves}, Annals of Mathematics Studies, 108. Princeton, New Jersey: Princeton University Press. XIV, 514 p. (1985)
\bibitem[KS08]{KS08} Kato K., Saito T.: \textit{Ramification theory for varieties over a perfect field}, Ann. Math. (2) 168, No. 1, 33--96 (2008)
\bibitem[KS13]{KS13} Kato K., Saito T.: \textit{Ramification theory for varieties over a local field}, Publ. Math., Inst. Hautes \'Etud. Sci. 117, 1--178 (2013)
\bibitem[Liu02]{Liu02} Liu Q.: \textit{Algebraic Geometry and Arithmetic Curves}, Vol. 6 of \textit{Oxford Graduate Texts in Mathematics}. Oxford Univ. Press, Oxford (2002)
\bibitem[LM83]{LM83} Lewin-M\'en\'egaux, Ren\'ee: \textit{Mod\`eles minimaux de torseurs}, C. R. Acad. Sci., Paris, S\'er. I 297, 257--260 (1983)
\bibitem[Ne67]{Ne67} N\'eron A.: \textit{Mod\`eles minimaux des espaces homog\`enes}, Proc. Conf. local Fields, NUFFIC Summer School Driebergen 1966, 66--77 (1967)
\bibitem[Ra67]{Ra67} Raynaud M.: \textit{Passage au quotient par une relation d'\'equivalence plate}, Proc. Conf. local Fields, NUFFIC Summer School Driebergen 1966, 78--85 (1967)
\bibitem[Rom12]{Rom12} Romagny M.: \textit{Effective models of group schemes}, J. Algebr. Geom. 21, No. 4, 643--682 (2012)
\bibitem[Sa12]{Sa12} Saito T.: \textit{Ramification of local fields with imperfect residue fields. III}, Math. Ann. 352, No. 3, 567--580 (2012)
\bibitem[Sa19]{Sa19} Saito T.: \textit{Ramification groups of coverings and valuations}, Tunis. J. Math. 1, No. 3, 373--426 (2019)
\bibitem[Se79]{Se79} Serre J.-P.: \textit{Local Fields}, Graduate Texts in Mathematics 67, Springer, New York (1979)
\bibitem[SGA1]{SGA1} Grothendieck A.: \textit{Rev\^etements \'etales et groupe fondamental}, S{\'e}minaire de g{\'e}om{\'e}trie alg{\'e}brique du Bois Marie 1960-61. Directed by A. Grothendieck. With two papers by M. Raynaud. Updated and annotated reprint of the 1971 original [Lecture Notes in Math., 224, Springer, Berlin; MR0354651]. Documents Math\'ematiques 3, Soci\'et\'e Math\'ematique de France, 2003.
\bibitem[SGA3-1]{SGA3-1} Demazure M., Grothendieck A.: \textit{S\'eminaire de g\'eom\'etrie alg\'ebrique du Bois Marie 1962-64. Sch\'emas en groupes (SGA 3). Tome I: Propri\'et\'es g\'en\'erales des sch\'emas en groupes}. Paris: Soci\'et\'e Math\'ematique de France (2011)
\bibitem[Za16]{Za16} Zalamansky G.: \textit{Ramification of inseparable coverings of schemes and application to diagonalizable group actions}, \href{https://arxiv.org/abs/1603.09284}{arXiv 1603.09284}.
\end{thebibliography}
\end{document}